\documentclass[11pt]{article}

\usepackage{amssymb,amsmath,amsthm}

\newcommand{\n}{\noindent}
\newcommand{\ovl}{\overline}
\newcommand{\intl}{\int\limits}
\newcommand{\bb}[1]{\mathbb{#1}}
\newcommand{\cl}[1]{\mathcal{#1}}
\newcommand{\vp}{\varepsilon}

\numberwithin{equation}{section}

\theoremstyle{plain}
\newtheorem{thm}{Theorem}[section]
\newtheorem{lem}{Lemma}[section]

\theoremstyle{remark}

\begin{document}

\title{Initial Blow-up of Solutions of Semilinear Parabolic Inequalities}

\author{Steven D. Taliaferro\\
Department of Mathematics\\
Texas A\&M University\\
College Station, TX\ 77843-3368\\
{\tt stalia@math.tamu.edu}}

\date{}
\maketitle

\thispagestyle{empty}

\begin{abstract}
We study classical nonnegative solutions $u(x,t)$ of the semilinear
parabolic inequalities
\[
0\le u_t-\Delta u\le u^p \qquad\text{in}\qquad \Omega\times
(0,1)
\] 
where $p$ is a positive constant and $\Omega$ is a smooth
bounded domain in ${\bb R}^n$, $n\ge 1$.

We show that a necessary and sufficient condition on $p$ for
such solutions $u$ to satisfy an a priori bound on compact subsets $K$
of $\Omega$ as $t\to 0^+$ is $p\le 1+2/n$ and in this case
the a priori bound on $u$ is
\[
\max_{x\in K} u(x,t)  = O(t^{-n/2}) \qquad\text{as}\qquad t\to 0^+.
\]

If in addition, $u$ satisfies Dirichlet boundary conditions $u=0$
on $\partial \Omega\times (0,1)$ and $p< 1+2/(n+1)$, then we
obtain a uniform a priori bound for $u$ on the entire set $\Omega$ as
$t\to 0^+$.
\medskip

\n {\it Keywords:} Initial blow-up, semilinear parabolic inequalities.
\medskip

\n 2010 Mathematics Subject Classification Codes: 35B09, 35B33, 35B40,
35B44, 35B45, 35K10, 35K58, 35R45.
\end{abstract}

\section{Introduction}\label{sec1}

\indent 

It is not hard to prove that if $u$ is a nonnegative solution of the
heat equation
\begin{equation}\label{neq1.1}
u_t-\Delta u=0 \qquad \hbox{in} \quad \Omega\times (0,1),   
\end{equation} 
where $\Omega$ is an open subset of ${\bb R}^n$, $n\ge 1$, then for
each compact subset $K$ of $\Omega$, we have 
\begin{equation}\label{neq1.2}
\max_{x\in K}\,u(x,t)=O(t^{-n/2}) \qquad \hbox{as}\quad t\to 0^+.
\end{equation} 
The exponent $-n/2$ in \eqref{neq1.2} is optimal because the Gaussian
\begin{equation}\label{neq1.3}
\Phi(x,t) = \left\{\begin{array}{ll}
\dfrac1{(4\pi t)^{n/2}} e^{-\frac{|x|^2}{4t}},&t>0\\
0,&t\le 0\end{array}\right.
\end{equation}
is a nonnegative solution of the heat equation in ${\bb R}^n\times
{\bb R}-\{(0,0)\}$ and
\begin{equation}\label{neq1.4}
\Phi(0,t)=(4\pi t)^{-n/2} \qquad \hbox{for}\quad t>0.
\end{equation}

It is also not hard to prove that if $u$ is a nonnegative solution of the
heat equation with Dirichlet boundary conditions
\begin{equation}\label{neq1.5}
\begin{split}
u_t-\Delta u&=0 \qquad \hbox{in} \quad \Omega\times
           (0,1)\\
           u&=0 \qquad \hbox{on} \quad \partial\Omega\times (0,1),
\end{split}
\end{equation}
where $\Omega$ is a $C^2$ bounded domain in ${\bb R}^n$, $n\ge 1$,
then there exists a positive constant $C$ such that
\begin{equation}\label{neq1.6}
u(x,t)\le C\frac{\frac{\rho(x)}{\sqrt{t}}\wedge 1}{\sqrt{t}^{n+1}}
\quad\text{for all }(x,t)\in\Omega\times(0,1/2)
\end{equation}
where $\rho(x)={\rm dist}(x,\partial\Omega)$.

Note that \eqref{neq1.6} is an a priori bound  for $u$ on the entire
set $\Omega$ rather than on compact subsets of $\Omega$.
As we discuss and state precisely in the paragraph after 
Theorem \ref{thm1.3}, the 
estimate \eqref{neq1.6} is optimal for $x$ near the boundary of
$\Omega$ and $t$ small.

In this paper, we generalize these results to nonnegative solutions
$u(x,t)$ of the inequalities
\begin{equation}\label{neq1.7}
0\le u_t-\Delta u\le f(u)\qquad\text{in}\qquad \Omega\times
(0,1)
\end{equation}
when the continuous function $f\colon [0,\infty)\to[0,\infty)$ is not
too large at infinity. Note that solutions of the heat equation
\eqref{neq1.1} satisfy \eqref{neq1.7}. Our first result deals with 
nonnegative solutions
$u$ of \eqref{neq1.7} when no boundary conditions are imposed on $u$.

\begin{thm}\label{thm1.1}
Suppose $u(x,t)$ is a $C^{2,1}$ nonnegative solution of
\begin{equation}\label{neq1.8}
 0 \le u_t - \Delta u \le (u +1)^{1+2/n}
\quad \text{in}\quad \Omega\times (0,1),
\end{equation}
where $\Omega$ is an open subset of ${\bb R}^n$, $n\ge 1$. Then, for
each compact subset $K$ of $\Omega$, $u$ satisfies \eqref{neq1.2}.
\end{thm}

We proved Theorem \ref{thm1.1} in \cite{T} with the strong added
assumption that, 
\begin{equation}\label{neq1.10}
\text{for some } x_0\in\Omega, \, u \text{ is continuous on }
(\Omega\times [0,1)) - \{(x_0,0)\}.
\end{equation}

Theorem \ref{thm1.1} is optimal in two ways. First, the exponent
$-n/2$ on $t$ in \eqref{neq1.2} cannot be improved because, as already
pointed out, the Gaussian \eqref{neq1.3} 
is a $C^\infty$ nonnegative solution of
the heat equation in ${\bb R}^n\times{\bb R} - \{(0,0)\}$ satisfying 
\eqref{neq1.4}.

And second, the exponent $1+2/n$ in \eqref{neq1.8} cannot
be increased by the following theorem in \cite{T}.

\begin{thm}\label{thm1.2}
  Let $p>1+2/n$ and $\psi\colon (0,1)\to
  (0,\infty)$ be a continuous function. Then there exists a $C^\infty$
  nonnegative solution $u(x,t)$ of
\[
 0 \le u_t -\Delta u \le u^p
\quad \text{in}\quad ({\bb R}^n\times {\bb R}) - \{(0,0)\}
\]
satisfying $u\equiv 0$ in ${\bb R}^n\times (-\infty,0)$ and 
\[
u(0,t)\ne O(\psi(t)) \quad \text{as}\quad t\to 0^+.
\]
\end{thm}

Our next result deals with nonnegative solutions of \eqref{neq1.7}
satisfying Dirichlet boundary conditions. 

\begin{thm}\label{thm1.3}
Suppose $u\in C^{2,1}(\ovl\Omega\times (0,1))$ is
a nonnegative solution of 
\begin{equation}\label{neq1.12}
\begin{split}
           0\le u_t-\Delta u\le (u + 1)^p \qquad &\hbox{in} 
           \quad \Omega\times (0,1)\\
           u=0 \qquad \qquad \qquad \qquad &\hbox{on} \quad 
           \partial\Omega\times (0,1),
\end{split}
\end{equation}
where $1<p<1+2/(n+1)$ and
$\Omega$ is a $C^2$ bounded domain in ${\bb R}^n$, $n\ge 1$. 
Then there exists a positive constant $C$ such that
$u$ satisfies \eqref{neq1.6}.
\end{thm}

Note that the bound \eqref{neq1.6} for $u$ in Theorem \ref{thm1.3}
is, like $u$, zero on
$\partial\Omega\times (0,1)$.
Furthermore, the estimate \eqref{neq1.6} is optimal for $x$ near the
boundary of $\Omega$ and $t$ small. More precisely,  
let 
$x_0\in\partial\Omega$, $G(x,y,t)$ be the heat kernel of the Dirichlet
Laplacian in $\Omega\times (0,1)$, and $\eta$ be the unit inward
normal to $\Omega$ at $x_0$. Then using the lower bound for $G$ in
\cite{Z}, it is easy to show that 
\[
u(x,t):=\lim_{r\to 0^+}\frac{G(x,x_0+r\eta,t)}{r}
\]
is a nonnegative solution of \eqref{neq1.5}, and hence of
\eqref{neq1.12}, such that for some $T>0$
\[
\frac{u(x,t)}{(\frac{\rho(x)}{\sqrt{t}}\wedge 1)/\sqrt{t}^{n+1}}
\]
is bounded between positive constants for all
$(x,t)\in\Omega\times(0,T)$ satisfying $|x-x_0|<\sqrt{t}$.

By modifying the proof of Theorem \ref{thm1.2}, it can be shown that
Theorem \ref{thm1.3} is not true for $p>1+2/n$. An open question is
for what values of $p\in[1+2/(n+1),1+2/n]$ is Theorem \ref{thm1.3}
true.  

Philippe Souplet communicated to us a proof of Theorem \ref{thm1.3} in the
special case that the differential inequalities in problem
\eqref{neq1.12} are replaced with the equation $u_t-\Delta
u=u^p$. However his method of proof does not seem to work for Theorem
\ref{thm1.3} as stated. See also \cite[Theorem 26.14(i)]{QS}.

Theorems \ref{thm1.1} and \ref{thm1.3} can be strengthened by
replacing the term 1 on the right sides of \eqref{neq1.8} and
\eqref{neq1.12} with a term which tends to infinity as $t\to 0^+$.
We state and prove
these strengthend versions of Theorems \ref{thm1.1} and \ref{thm1.3}
in Sections \ref{sec3} and \ref{sec4} respectively.

The proofs of Theorems \ref{thm1.1} and \ref{thm1.3} rely heavily on
Lemmas \ref{lem2.1} and \ref{lem2.2}, respectively, which we state and
prove in Section \ref{sec2}. We are able to prove Theorem \ref{thm1.1}
without condition \eqref{neq1.10} because we do not impose this kind
of condition on the function $u$ in Lemma \ref{lem2.1}.

As in \cite{T}, a crucial step
in the proofs of Theorems \ref{thm1.1} and \ref{thm1.3} is an
adaptation and extension to parabolic inequalities of a method of
Brezis \cite{B} concerning elliptic equations and based on Moser's
iteration. This method is used to obtain an estimate of the form 
\[
\|u_j\|_{L^{\frac{n+2}{n}q}(\Omega')}
\le C\|u_j\|_{L^{q}(\Omega^)}
\]
where $q>1$, $\Omega'\subset\Omega$, 
$C$ is a constant which does not depend on $j$,
and $u_j$, $j=1,2,\dots$, is obtained from the
function $u$ in Theorem \ref{thm1.1} or \ref{thm1.3}
by appropriately scaling $u$ about 
$(x_j,t_j)$ where $(x_j,t_j)\in\Omega\times(0,1)$ is a sequence 
such that $t_j\to 0^+$ and for
which \eqref{neq1.2} or \eqref{neq1.6} is violated.

Our proofs also rely on upper and lower bounds for the heat kernel of
the Dirichlet Laplacian. We use the upper bound in \cite{H} and the
lower one in \cite{Z}.

The blow-up of solutions of the {\it equation} 
\begin{equation}\label{neq1.14}
u_t - \Delta u = u^p
\end{equation}
has been extensively studied in \cite{AD, AHV, BV, BF, BPT, GK, G, HV, K, MM,
  MZ, M, O, PQS, PY, QSW, V} and elsewhere. The book \cite{QS} is an excellent
  reference for many of these results. However, other than
  \cite{T}, we know of no previous blow-up results for the {\it inequalities}
\[
 0 \le u_t -\Delta u \le u^p.
\]
Also, blow-up of solutions of 
$au^p \le u_t -\Delta u \le u^p
$,
where $a\in(0,1)$, has been studied in \cite{T1}.

\section{Preliminary lemmas}\label{sec2}

\indent  

For the proof in Section \ref{sec3} of Theorem \ref{thm1.1}, we will
need the following lemma.

\begin{lem}\label{lem2.1}
Suppose $u$ is a $C^{2,1}$ nonnegative solution of
\begin{equation}\label{eq2.1}
 Hu\ge 0\quad \text{in}\quad B_4(0)\times (0,3) 
\subset {\bb R}^n \times {\bb R},\qquad n\ge 1,
\end{equation}
where $Hu = u_t-\Delta u$ is the heat operator. Then
\begin{equation}\label{eq2.2}
 u,Hu \in L^1(B_2(0) \times (0,2))
\end{equation}
and there exist a finite positive Borel measure $\mu$ on
$B_2(0)$ and $h\in C^{2,1}(B_1(0) \times (-1,1))$ satisfying
\begin{alignat}{2}
 \label{eq2.3}
Hh &= 0  \quad \text{in} &\quad &B_1(0) \times (-1,1)\\
\label{eq2.4}
h &= 0 \quad \text{in} &\quad &B_1(0)\times (-1,0] 
\end{alignat}
such that
\begin{equation}\label{eq2.5}
 u = N +v+h\quad \text{in}\quad B_1(0)\times (0,1)
\end{equation}
where
\begin{align}\label{eq2.6}
 N(x,t) &:= \int^2_0 \intl_{|y|<2} \Phi(x-y,t-s) Hu(y,s)\,dy\,ds,\\
\label{eq2.7}
v(x,t) &:= \intl_{|y|<2} \Phi(x-y,t)\,d\mu(y),
\end{align}
and $\Phi$ is the Gaussian \eqref{neq1.3}.
\end{lem}

\begin{proof}
Let $\varphi_1\in C^2(\ovl{B_3(0)})$ and $\lambda>0$ 
satisfy
\[
 \begin{array}{cl}
\left.\begin{array}{c} -\Delta\varphi_1= \lambda\varphi_1\\ 
\varphi_1>0\end{array}\right\}&\text{for } |x|<3\\
\varphi_1=0&\text{for } |x|=3.
 \end{array}
\]
Then for $0<t\le 2$, we have by \eqref{eq2.1} that
\begin{align*}
 0 &\le \intl_{|x|<3} [Hu(x,t)]\varphi_1(x)\,dx\\
&= \intl_{|x|<3} u_t(x,t) \varphi_1(x)\,dx 
+ \lambda \intl_{|x|<3} u(x,t) \varphi_1(x)\,dx 
+ \intl_{|x|=3} u(x,t) \frac{\partial\varphi_1(x)}{\partial \eta} \,dS_x\\
&\le U'(t) + \lambda U(t)
\end{align*}
where $U(t) = \intl_{|x|<3} u(x,t) \varphi_1(x)\,dx$. Thus
$(U(t)e^{\lambda t})'\ge 0$ for $0<t\le 2$ and consequently for some
$U_0\in [0,\infty)$ we have
\begin{equation}\label{eq2.9}
 U(t) = (U(t)e^{\lambda t}) e^{-\lambda t}\to U_0
\quad \text{as}\quad t\to 0^+.
\end{equation}
Thus $u\varphi_1 \in L^1(B_3(0)\times (0,2))$. 
Hence, since for $0<t\le 2$,
\begin{align}
  \int^2_t \intl_{|x|<3} Hu(x,\tau) \varphi_1(x)\,dx\,d\tau &=
  \intl_{|x|<3} 
\left(\int^2_t u_t(x,\tau)\,d\tau\right)\varphi_1(x)\,dx 
- \int^2_t \intl_{|x|<3} (\Delta u(x,\tau)) \varphi_1(x) \,dx \,d\tau\notag\\
  &= \intl_{|x|<3} u(x,2) \varphi_1(x)\,dx 
- \intl_{|x|<3} u(x,t) \varphi_1(x)\,dx\notag\\
  &\quad + \int^2_t \intl_{|x|=3} u(x,\tau) 
\frac{\partial\varphi_1(x)}{\partial \eta} \,dS_x\,d\tau \notag\\
\label{eq2.8}
&\quad +\lambda\int^2_t \intl_{|x|<3} u(x,\tau) \varphi_1(x) \,dx\,d\tau,
\end{align}
we see that $(Hu)\varphi_1 \in
L^1(B_3(0)\times (0,2))$. So \eqref{eq2.2} holds.

By \eqref{eq2.9},
\begin{equation}\label{eq2.10}
 \intl_{|x|\le 2} u(x,t)\,dx\quad \text{is bounded for}\quad 0<t\le 2.
\end{equation}
Hence there exists a finite positive Borel measure $\hat\mu$ on
$\ovl{B_2(0)}$ and a sequence $t_j$ decreasing to $0$ such that
for all $g\in C(\ovl{B_2(0)})$ we have
\[
 \intl_{|x|\le 2} g(x) u(x,t_j)\,dx \longrightarrow \intl_{|x|\le 2} 
g(x)\,d\hat\mu\quad \text{as}\quad j\to\infty.
\]
In particular, for all $\varphi\in C^\infty_0(B_2(0))$ we have
\begin{equation}\label{eq2.11}
 \intl_{|x|<2} \varphi(x) u(x,t_j) \,dx
\longrightarrow \intl_{|x|<2} \varphi(x)\,d\mu\quad \text{as}
\quad j\to \infty.
\end{equation}
where we define $\mu$ to be the restriction of $\hat\mu$ to $B_2(0)$.

For $(x,t)\in {\bb R}^n\times (0,\infty)$, let $v(x,t)$ be defined by
\eqref{eq2.7}. Then $v\in C^{2,1}({\bb R}^n \times (0,\infty))$,
$Hv=0$ in ${\bb R}^n\times (0,\infty)$, and
\begin{equation}\label{eq2.13}
 \intl_{{\bb R}^n} v(x,t)\,dx = \intl_{|y|<2} d\mu(y) 
< \infty\quad \text{for}\quad t>0.
\end{equation}
Thus $v\in L^1({\bb R}^n\times (0,2))$.

For $\varphi\in C^\infty_0(B_2(0))$ and $t>0$ we have
\[
 \intl_{|x|<2} \varphi(x) v(x,t)\,dx 
= \intl_{|y|<2} \left(\, \intl_{|x|<2} \Phi(x-y,t) 
\varphi(x)\,dx\right) d\mu(y)\\
 \longrightarrow \intl_{|y|<2} \varphi(y)\,d\mu(y)
\quad \text{as}\quad t\to 0^+,
\]
and hence it follows from \eqref{eq2.11} that
\begin{equation}\label{eq2.14}
 \intl_{|x|<2} \varphi(x) (u(x,t_j) - v(x,t_j))\,dx \to 0 
\quad \text{as}\quad j\to \infty.
\end{equation}

Let
\[
 f:= \begin{cases}
      Hu,&\text{in $B_2(0)\times (0,2)$}\\
0,&\text{elsewhere in ${\bb R}^n\times {\bb R}$.}
     \end{cases}
\]
Then by \eqref{eq2.2},
\begin{equation}\label{eq2.15}
 f\in L^1({\bb R}^n\times {\bb R}).
\end{equation}
Let
\[
 w:= \begin{cases}
       u-v,&\text{in $B_2(0)\times (0,2)$}\\
0,&\text{elsewhere in ${\bb R}^n\times {\bb R}$.}
      \end{cases}
\]
Then
\begin{align}
 \label{eq2.16}
&w\in C^{2,1}(B_2(0) \times (0,2))\cap L^1({\bb R}^n\times {\bb R}),\\
&Hw = f \quad \text{in}\quad B_2(0) \times (0,2),\notag
\end{align}
and
\begin{equation}\label{eq2.17}
\intl_{|x|<2} |w(x,t)|\,dx\quad\text{is bounded for}\quad 0<t<2
\end{equation}
by \eqref{eq2.10} and \eqref{eq2.13}. Let $\Omega = B_1(0) \times
(-1,1)$ and define $\Lambda\in {\cl D}'(\Omega)$ by $\Lambda = -Hw+f$,
that is
\[
 \Lambda\varphi = \int wH^*\varphi + \int f\varphi
\quad \text{for}\quad \varphi\in C^\infty_0(\Omega),
\]
where $H^*\varphi := \varphi_t + \Delta\varphi$. We now show
$\Lambda=0$. 
Let $\varphi\in C^\infty_0(\Omega)$, let $j$ be a fixed positive integer, and
let $\psi_\vp\colon {\bb R}\to [0,1]$, $\vp$ small and positive, be a
one parameter family of smooth nondecreasing functions such that 
\[
\psi_\vp(t)= \begin{cases}
                          1,&t>t_j+\vp\\
                          0,&t<t_j-\vp.
             \end{cases}
\]  
where $t_j$ is as in \eqref{eq2.11}.
Then
\begin{align*}
-\int f\varphi\psi_\vp 
&= -\int (Hw) \varphi\psi_\vp = \int wH^*(\varphi\psi_\vp)\\
&= \int w(\varphi_t\psi_\vp + \varphi\psi'_\vp + \psi_\vp \Delta\varphi)\\
&= \int w\psi_\vp H^*\varphi + \int w\varphi\psi'_\vp.
\end{align*}
Letting $\vp\to 0^+$ we get
\begin{equation}\label{eq2.18}
 -\int^1_{t_j} \intl_{|x|<1} f\varphi\,dx\,dt 
= \int^1_{t_j} \intl_{|x|<1} wH^* \varphi\,dx\,dt 
+ \intl_{|x|<1} w(x,t_j) \varphi(x,t_j)\,dx.
\end{equation}
Also, it follows from \eqref{eq2.17} and \eqref{eq2.14} that
\begin{align*}
\intl_{|x|<1} w(x,t_j) \varphi(x,t_j)\,dx 
&= \intl_{|x|<1} w(x,t_j) [\varphi(x,t_j) - \varphi(x,0)] \,dx
+\intl_{|x|<1} w(x,t_j) \varphi(x,0)\,dx\\ 
&\to 0 \quad \text{as} \quad j\to \infty.
\end{align*}
Thus letting $j\to \infty$ in \eqref{eq2.18} and using \eqref{eq2.15} and
\eqref{eq2.16} we get $-\int f\varphi = \int wH^*\varphi$. So
$\Lambda=0$. 

For $(x,t) \in {\bb R}^n\times {\bb R}$, let $N(x,t)$ be
defined by \eqref{eq2.6}. Then
\[
 N(x,t) = \iint\limits_{{\bb R}^n\times {\bb R}} \Phi(x-y,t-s) f(y,s)\,dy\,ds
\]
and $N\equiv 0$ in ${\bb R}^n\times (-\infty,0)$. By \eqref{eq2.15},
we have $N\in L^1(\Omega)$ and $HN=f$ in ${\cl D}'(\Omega)$. Thus
\[
 H(w-N) = -\Lambda+f-f =0 \quad \text{in}\quad {\cl D}'(\Omega)
\]
which implies
\[
 w-N=h\quad \text{in}\quad {\cl D}'(\Omega)
\]
for some $C^{2,1}$ solution $h$ of \eqref{eq2.3} and
\eqref{eq2.4}. Hence \eqref{eq2.5} holds.
\end{proof}

For the proof in Section \ref{sec4} of Theorem \ref{thm1.3}, we will
need the following lemma.

\begin{lem}\label{lem2.2}
Suppose $u\in C^{2,1}(\ovl\Omega\times (0,2T))$ is a nonnegative solution of
\[
Hu \ge 0 \qquad \text{\rm in}\qquad \Omega\times (0,2T),
\]
where $Hu=u_t-\Delta u$ is the heat operator, $T$ is a positive
constant, and $\Omega$ is a bounded $C^2$ domain in ${\bb R}^n$, $n\ge
1$. Then
\begin{equation}\label{eq2.19}
 u,\rho Hu \in L^1(\Omega\times (0,T)),
\end{equation}
where $\rho(x) = \text{\rm dist}(x,\partial\Omega))$. Moreover,
there exists $C>0$ such that 
\begin{equation}\label{eq2.20}
\begin{split}
0&\le u(x,t) - \int^t_0 \int_\Omega G(x,y,t-s) Hu(y,s)\,dy\,ds\\ 
&\le C\frac{\frac{\rho(x)}{\sqrt{t}} \wedge 1}{t^{\frac{n+1}2}}
+ \sup_{\partial\Omega\times(0,T)}u
\qquad \text{for all } (x,t)\in\Omega\times(0,T),
\end{split}
\end{equation}
where $G$ is the Dirichlet heat kernel for $\Omega$.
\end{lem}

\begin{proof}
  For $\varphi\in C^2(\Omega)\cap C^1(\ovl\Omega)$, $\varphi=0$ 
on $\partial\Omega$,  and $0<t<T$ we have
\begin{align}
  \int^T_t \int_\Omega [Hu(y,\tau)] \varphi(y) \,dy\,d\tau &= \int_\Omega
  u(y,T) 
\varphi(y) \,dy - \int_\Omega u(y,t) \varphi(y)\,dy\notag\\
\label{eq2.21}
&\quad -\int^T_t \int_\Omega u(y,\tau) \Delta\varphi(y)\,dy\,d\tau
+\int_t^T\int_{\partial\Omega}u(y,\tau)\frac{\partial\varphi(y)}
{\partial\eta}\,dS_y\,d\tau
\end{align}

Let $\varphi_1\in C^2(\Omega)\cap C^1(\ovl\Omega)$ and $\lambda>0$ satisfy 
\[
\begin{array}{cl}
\left.\begin{array}{r}
-\Delta\varphi_1=\lambda\varphi_1\\ 0<\varphi_1<1
\end{array}\right\}&\text{in } \Omega\\
\varphi_1=0&\text{on } \partial\Omega.
\end{array}
\]
Then for $0<t<2T$ we have
\begin{align*}
0 \le \int_\Omega Hu(y,t) \varphi_1(y)\,dy &= U'(t) + \lambda U(t)
+ \int_{\partial\Omega}u(y,t)\frac{\partial\varphi_1(y)}
{\partial\eta}\,dS_y\\
&\le U'(t) + \lambda U(t),
\end{align*}
where $U(t) = \int_\Omega u(y,t) \varphi_1(y)\,dy$. Thus
$(U(t)e^{\lambda t})'\ge 0$ for $0<t<2T$ and hence for some $U_0\ge 0$
we have
\begin{equation}\label{eq2.22}
 U(t) = (U(t)e^{\lambda t}) e^{-\lambda t} \to U_0 \quad \text{as}\quad t\to 0^+.
\end{equation}
Consequently $u\varphi_1 \in L^1(\Omega\times (0,T))$. So taking
$\varphi=\varphi_1$ in \eqref{eq2.21} we have
\begin{equation}\label{eq2.23}
 \varphi_1 Hu\in L^1(\Omega\times (0,T)),
\end{equation}
and taking $\varphi=\varphi^2_1$ in \eqref{eq2.21} we obtain
$u|\nabla\varphi_1|^2 \in L^1(\Omega\times (0,T))$. Thus, since
$\varphi_1+|\nabla\varphi_1|^2$ is bounded away from zero on
$\ovl\Omega$, we have $u\in L^1(\Omega\times(0,T))$. Hence, since
$\varphi_1/\rho$ is bounded between positive constants on $\Omega$, it
follows from \eqref{eq2.23} that \eqref{eq2.19} holds, and by
\eqref{eq2.22} we have
\begin{equation}\label{eq2.24}
 \int_\Omega u(y,t) \rho(y)\,dy \quad \text{is bounded for}\quad 0<t\le T.
\end{equation}

Let $x\in\Omega$ and $0<\tau<t<T$ be fixed. Then for $\vp>0$ we have
\begin{align}\label{eq2.25}
\begin{split}
\int_\Omega &G(x,y,\vp) u(y,t)\,dy  
 - \int^t_\tau\int_\Omega G(x,y,t+\vp-s) Hu(y,s)\,dy\,ds\\
&=\int_\Omega G(x,y,t+\vp-\tau) u(y,\tau)\,dy
 - \int_\tau^t\int_{\partial\Omega}u(y,s)\frac{\partial G(x,y,t+\vp-s)}
{\partial\eta_y}\,dS_y\,ds\\
&\ge 0.
\end{split}
\end{align}
Since $\int_\Omega G(x,y,\zeta)\,dy\le 1$ for $\zeta>0$, we have
\begin{align*}
0&\le - \int_\tau^t\int_{\partial\Omega}\frac{\partial G(x,y,t+\vp-s)}
{\partial\eta_y}\,dS_y\,ds\\
&=\int_\Omega G(x,y,\vp) \,dy - \int_\Omega G(x,y,t+\vp-\tau)\,dy
\le 1
\end{align*} 
and
\[
\int_\Omega G(x,y,t+\vp-s) Hu(y,s)\,dy \le \max_{\ovl\Omega \times
  [\tau,t]} Hu<\infty
\]
for $\vp>0$ and $\tau\le s\le t$. Thus, letting $\vp\to 0^+$ in
\eqref{eq2.25} and using the fact that the function $(y,\zeta)\to
G(x,y,\zeta)$ is continuous for $(y,\zeta) \in\ovl\Omega \times
(0,\infty)$ we get
\begin{align}\notag  
0&\le u(x,t) -  \int^t_\tau \int_\Omega G(x,y,t-s) Hu(y,s)\,dy\,ds\\
\label{eq2.26}
&\le v(x,t,\tau) + \sup_{\partial\Omega\times(0,T)}u
\end{align}
where
\[
v(x,t,\tau) := \int_\Omega G(x,y,t-\tau)u(y,\tau)\,dy
\le C\frac{\frac{\rho(x)}{\sqrt{t-\tau}} \wedge 1}
{(t-\tau)^{\frac{n+1}2}} \int_\Omega u(y,\tau)\rho(y)\,dy
\]
because, as shown by Hui \cite[Lemma 1.3]{H}, there exists a positive
constant $C=C(n,\Omega,T)$ such that if 
\[
\widehat G(r,t) = \frac{C}{t^{n/2}} e^{-r^2/(Ct)} \quad
\text{for}\quad r\ge 0 \quad \text{and}\quad t>0
\]
then the heat kernel $G(x,y,t)$ for $\Omega$ satisfies
\begin{equation}\label{eq2.27}
  G(x,y,t) \le \left(\frac{\rho(x)}{\sqrt t} \wedge 1\right) 
\left(\frac{\rho(y)}{\sqrt t} \wedge 1\right) \widehat G(|x-y|,t)
\quad \text{for}\quad x,y\in\Omega\quad \text{and}\quad  0<t\le T.
\end{equation}
Hence, letting $\tau\to 0^+$ in \eqref{eq2.26} and using \eqref{eq2.24}
and the monotone convergence theorem we obtain \eqref{eq2.20}.
\end{proof}

For the proofs in Sections \ref{sec3} and \ref{sec4} of Theorems
\ref{thm1.1} and \ref{thm1.3} respectively we will need the following
lemma whose proof is an adaptation to parabolic inequalities of a
method of Brezis \cite{B} for elliptic equations.

\begin{lem}\label{lem2.3}
  Suppose $T>0$ and $\lambda>1$ are constants, $B$ is an open ball in
  ${\bb R}^n$, $E=B\times (-T,0)$, and $\varphi\in C^\infty_0(B\times
  (-T,\infty))$. Then there exists a positive constant $C$ depending
  only on
  \begin{equation}\label{eq1}
n,\lambda,\quad \text{and}\quad \sup_E\left(|\varphi|,
  |\nabla\varphi|, 
\left|\frac{\partial\varphi}{\partial t}\right|, |\Delta\varphi|\right)
\end{equation}
such that if $\Omega$ is a $C^2$ bounded domain in ${\bb R}^n$,
$\Omega\cap B \ne\emptyset$, $D = \Omega \times (-T,0)$, and $u\in
C^{2,1}(\ovl D)$ is a nonnegative solution of
\begin{alignat*}{2}
Hu &\ge 0 &\quad &\text{in } \Omega\times (-T,0)\\
u &= 0&\quad &\text{on } (\partial\Omega\cap B)\times (-T,0)
\end{alignat*}
then
\begin{equation}\label{eq2}
\left(~\iint\limits_{E\cap D} (u^\lambda\varphi^2)^{\frac{n+2}n} 
\, dx\, dt\right)^{\frac{n}{n+2}} 
\le C \left(~\iint\limits_{E\cap D} (Hu)u^{\lambda-1} \varphi^2\,dx\,dt 
+ \iint\limits_{E\cap D} u^\lambda \, dx\, dt\right).
\end{equation}
\end{lem}

\begin{proof}
Let $u$ be as in the lemma. Since 
\begin{equation}\label{eq3}
\nabla u \cdot \nabla(u^{\lambda-1}\varphi^2) = \frac{4(\lambda-1)}{\lambda^2} |\nabla(u^{\lambda/2}\varphi)|^2 - \frac{\lambda-2}{\lambda^2} \nabla u^\lambda \cdot \nabla\varphi^2 -  \frac{4(\lambda-1)}{\lambda^2} u^\lambda|\nabla\varphi|^2
\end{equation}
we have for $-T<t<0$ that
\begin{align}\label{eq4}
\int\limits_{B\cap\Omega} (-\Delta u)u^{\lambda-1}\varphi^2\, dx &= \int\limits_{B\cap\Omega} \nabla u \cdot \nabla(u^{\lambda-1}\varphi^2)\, dx\notag\\
&\ge \frac{4(\lambda-1)}{\lambda^2} \int\limits_{B\cap\Omega} |\nabla(u^{\lambda/2}\varphi)|^2\, dx - C \int\limits_{B\cap\Omega} u^\lambda\, dx
\end{align}
where $C$ is a positive constant depending only on the quantities \eqref{eq1} whose value may change from line to line. Also, for $x\in B\cap\Omega$ we have
\begin{align}
\int^0_{-T} u_t u^{\lambda-1}\varphi^2\, dt &= \frac1\lambda \int^0_{-T} \frac{\partial u^\lambda}{\partial t}\varphi^2\ dt\notag\\
&= \frac1\lambda \left[u(x,0)^\lambda \varphi(x,0)^2 - \int^0_{-T} u^\lambda \frac{\partial\varphi^2}{\partial t}\, dt\right]\notag\\
\label{eq5}
&\ge - C \int^0_{-T} u^\lambda\, dt.
\end{align}
Integrating inequality \eqref{eq4} with respect to $t$ from $-T$ to 0, integrating inequality \eqref{eq5} with respect to $x$ over $B\cap\Omega$, and then adding the two resulting inequalities we get
\begin{equation}\label{eq6}
C(I+B)\ge \iint\limits_{E\cap D} |\nabla(u^{\lambda/2}\varphi)|^2\, dx\, dt
\end{equation}
where
\[
 I = \iint\limits_{E\cap D} (Hu)u^{\lambda-1}\varphi^2\, dx\, dt\quad \text{and}\quad B = \iint\limits_{E\cap D} u^\lambda\, dx\, dt.
\]
Multiplying \eqref{eq6} by
\[
 M := \max_{-T\le t\le 0} \left(~\int\limits_{B\cap\Omega} u^\lambda \varphi^2\, dx\right)^{2/n}
\]
and using the parabolic Sobolev inequality (see \cite[Theorem
6.9]{Lie}) 
we obtain
\begin{equation}\label{eq7}
 C(I+B)M \ge A := \iint\limits_{E\cap D} (u^\lambda\varphi^2)^{\frac{n+2}n}\, dx\, dt.
\end{equation}
Since 
\begin{align*}
 \frac\partial{\partial t} (u^\lambda\varphi^2) &= \lambda u^{\lambda-1}u_t\varphi^2 + 2u^\lambda\, \varphi \varphi_t\\
&= \lambda u^{\lambda-1} \varphi^2(\Delta u + Hu) + 2u^\lambda \varphi\varphi_t
\end{align*}
it follows from \eqref{eq4} that for $-T<t<0$ we have
\[
 \frac\partial{\partial t} \int\limits_{B\cap\Omega} u^\lambda \varphi^2\, dx \le C \int\limits_{B\cap\Omega} u^\lambda\, dx + \lambda \int\limits_{B\cap\Omega} u^{\lambda-1} \varphi^2 Hu\, dx
\]
and thus
\begin{equation}\label{eq8}
 M^{\frac{n}2} \le C(I+B).
\end{equation}
Substituting \eqref{eq8} in \eqref{eq7} we get
\[
A\le C(I+B)^{\frac{n+2}n}
\]
which implies \eqref{eq2}.
\end{proof}

\section{Proof of Theorem \ref{thm1.1}}\label{sec3}

\indent 

In this section we prove the following theorem which clearly implies
Theorem \ref{thm1.1}.

\begin{thm}\label{thm3.1}
Suppose $u$ is a $C^{2,1}$ nonnegative solution of
\begin{equation}\label{eq3.1}
 0 \le u_t-\Delta u \le b\left(u + \frac{1}{\sqrt{t}^n}\right)^{1+2/n}
\quad \text{in}\quad \Omega\times (0,T),
\end{equation}
where $T$ and $b$ are positive constants and $\Omega$ is an open
subset of ${\bb R}^n$, $n\ge 1$. Then, for each compact subset $K$ of
$\Omega$, we have
\begin{equation}\label{eq3.2}
 \max_{x\in K} u(x,t) = O(t^{-n/2}) \quad \text{as}\quad t\to 0^+.
\end{equation}
\end{thm}

\begin{proof}
  To prove Theorem \ref{thm3.1}, we claim it suffices to prove Theorem
  \ref{thm3.1}$'$ where Theorem \ref{thm3.1}$'$ is the theorem
  obtained from Theorem \ref{thm3.1} by replacing \eqref{eq3.1} with
\begin{equation}\label{eq3.3}
 0 \le u_t - \Delta u \le\left(u + \frac{b}{\sqrt{t}^n}\right)^{1+2/n}
\quad \text{in}\quad B_4(0) \times (0,3)
\end{equation}
and replacing \eqref{eq3.2} with
\begin{equation}\label{eq3.4}
 \max_{|x|\le \frac12} u(x,t) = O(t^{-n/2})
\quad \text{as}\quad t\to 0^+.
\end{equation}
To see this, let $u$ be as in Theorem \ref{thm3.1} and let $K$ be a
compact subset of $\Omega$. Since $K$ is compact there exist finite
sequences $\{r_j\}^N_{j=1} \subset (0,\sqrt T/4)$ and $\{x_j\}^N_{j=1}
\subset K$ such that
\[
 K \subset \bigcup^N_{j=1} B_{r_j/2} (x_j) 
\subset \bigcup^N_{j=1} B_{4r_j}(x_j) \subset\Omega.
\]
Let $v_j(y,s) = r_j^nb^{n/2}u(x,t)$, where $x=x_j+r_jy$ and $t=r^2_js$. Then
\[
 0 \le Hv_j \le\left(v_j + \frac{b^{n/2}}{\sqrt{s}^n}\right)^{1+2/n}
\quad \text{for}\quad |y|<4, \quad 0<s<16,
\]
where $Hv_j:=\frac{\partial v_j}{\partial s} - \Delta_yv_j$.
Hence by Theorem \ref{thm3.1}$'$ there exist $s_j\in (0,16)$ and 
$C_j>0$ such that
\[
 \max_{|y|\le\frac12} v_j(y,s) \le C_js^{-n/2}\quad \text{for}\quad 0<s<s_j.
\]
That is
\[
 \max_{|x-x_j|\le r_j/2} u(x,t) \le C_jb^{-n/2} t^{-n/2}
\quad \text{for}\quad 0<t<t_j := r^2_js_j.
\]
So for $0<t<\min\limits_{1\le j\le N} t_j$ we have
\begin{align*}
 \max_{x\in K} u(x,t) &\le \max_{1\le j\le N} \max_{|x-x_j|\le r_j/2} u(x,t)\\
&\le (\max_{1\le j\le N} C_j) b^{-n/2} t^{-n/2}.
\end{align*}
That is, \eqref{eq3.2} holds.

We now complete the proof of Theorem \ref{thm3.1} by proving Theorem
\ref{thm3.1}$'$. Suppose $u$ is a $C^{2,1}$ nonnegative solution of
\eqref{eq3.3}. By Lemma \ref{lem2.1},
\begin{align}\label{eq3.5}
&u,Hu\in L^1(B_2(0) \times (0,2))\\
\intertext{and}
\label{eq3.6}
&u = N+v+h\quad \text{in}\quad B_1(0) \times (0,1)
\end{align}
where $N,v$, and $h$ are as in Lemma \ref{lem2.1}.

Suppose for contradiction that \eqref{eq3.4} does not hold. Then there
exists a sequence $\{(x_j,t_j)\} \subset \ovl{B_{1/2}(0)} \times
(0,1/4)$ such that for some $x_0\in \ovl{B_{1/2}(0)}$ we have
$(x_j,t_j)\to (x_0,0)$ as $j\to\infty$ and
\begin{equation}\label{eq3.7}
 \lim_{j\to\infty} t^{n/2}_j u(x_j,t_j) = \infty.
\end{equation}
Clearly
\begin{equation}\label{eq3.8}
(4\pi t)^{n/2} v(x,t) \le \intl_{|y|<2} d\mu(y)<\infty
\quad \text{for}\quad (x,t)\in {\bb R}^n \times (0,\infty).
\end{equation}
For $(x,t)\in {\bb R}^n\times {\bb R}$ and $r>0$, let
\[
 E_r(x,t) := \{(y,s)\in {\bb R}^n\times {\bb R}\colon \ |y-x| < \sqrt
 r
\quad \text{and}\quad t-r<s<t\}.
\]

In what follows, the variables $(x,t)$ and $(\xi,\tau)$ are related by
\begin{equation}\label{eq3.9}
 x=x_j + \sqrt{t_j}\, \xi\quad \text{and}\quad t=t_j+t_j\tau
\end{equation}
and the variables $(y,s)$ and $(\eta,\zeta)$ are related by
\begin{equation}\label{eq3.10}
y = x_j + \sqrt{t_j}\, \eta\quad \text{and}\quad s=t_j+t_j\zeta.
\end{equation}
For each positive integer $j$, define
\begin{equation}\label{eq3.11}
 f_j(\eta,\zeta) = \sqrt{t_j}{}^{n+2} Hu(y,s)\quad \text{for} \quad (y,s)\in E_{t_j}(x_j,t_j)
\end{equation}
and define
\begin{equation}\label{eq3.12}
 u_j(\xi,\tau) = \sqrt{t_j}{}^n \iint\limits_{E_{t_j}(x_j,t_j)} \Phi(x-y,t-s) Hu(y,s)\, dy\,ds\quad \text{for}\quad (x,t)\in {\bb R}^n\times (0,\infty).
\end{equation}
By \eqref{eq3.5} we have
\begin{equation}\label{eq3.13}
 \iint\limits_{E_{t_j}(x_j,t_j)} Hu(y,s)\, dy\, ds\to 0\quad \text{as}\quad j\to\infty
\end{equation}
and thus making the change of variables \eqref{eq3.10} in \eqref{eq3.13} and using \eqref{eq3.11} we get
\begin{equation}\label{eq3.14}
 \iint\limits_{E_1(0,0)} f_j(\eta,\zeta) \, d\eta\, d\zeta \to 0\quad \text{as}\quad j\to\infty.
\end{equation}
Since
\[
 \Phi(x-y,t-s) = \frac{1}{\sqrt{t_j}{}^n} \Phi(\xi-\eta,\tau-\zeta)
\]
it follows from \eqref{eq3.12} and \eqref{eq3.11} that
\begin{equation}\label{eq3.15}
 u_j(\xi,\tau) = \iint\limits_{E_1(0,0)} \Phi(\xi-\eta,\tau-\zeta) f_j(\eta,\zeta) \, d\eta\, d\zeta.
\end{equation}
It is easy to check that for $1<q<\frac{n+2}n$ and $(\xi,\tau)\in {\bb R}^n \times (-1,0]$ we have
\begin{equation}\label{eq3.16}
 \left(~\iint\limits_{{\bb R}^n\times (-1,0)} \Phi(\xi-\eta, \tau-\zeta)^q\, d\eta\, d\zeta\right)^{1/q} < C(n,q) < \infty.
\end{equation}
Thus for $1 < q < \frac{n+2}n$ we have by \eqref{eq3.15} and standard $L^p$ estimates for the convolution of two functions that
\begin{equation}\label{eq3.17}
 \|u_j\|_{L^q(E_1(0,0))} \le C(n,q) \|f_j\|_{L^1(E_1(0,0))}\to 0\quad \text{as}\quad j\to \infty
\end{equation}
by \eqref{eq3.14}.
If 
\begin{equation}\label{eq3.18}
(x,t)\in \ovl{E_{t_j/4}(x_j,t_j)} \quad \text{and}\quad (y,s) \in {\bb R}^n\times (0,\infty) - E_{t_j}(x_j,t_j)
\end{equation}
then
\[
\Phi(x-y,t-s) \le \max_{0\le\tau<\infty} \Phi\left(\frac{\sqrt{t_j}}2,\tau\right) \le \frac{C(n)}{\sqrt{t_j}{}^n}.
\]
Thus for $(x,t)\in \ovl{E_{t_j/4}(x_j,t_j)}$ we have
\[
 \iint\limits_{B_2(0)\times (0,2)-E_j(x_j,t_j)} \Phi(x-y,t-s) Hu(y,s)\, dy\, ds\le \frac{C(n)}{\sqrt{t_j}{}^n} \iint\limits_{B_2(0)\times (0,2)} Hu(y,s)\, dy\, ds.
\]
It follows therefore from \eqref{eq3.6}, \eqref{eq3.8}, \eqref{eq3.5}, 
and \eqref{eq3.12} that
\begin{equation}\label{eq3.19}
u(x,t) \le \frac{u_j(\xi,\tau)+C}{\sqrt{t_j}{}^n}\quad \text{for}\quad (x,t) \in \ovl{E_{t_j/4}(x_j,t_j)}
\end{equation}
where $C$ is a positive constant which does not depend on $j$ or $(x,t)$.

Substituting $(x,t) = (x_j,t_j)$ in \eqref{eq3.19} and using \eqref{eq3.7} we obtain
\begin{equation}\label{eq3.20}
u_j(0,0)\to \infty\quad \text{as}\quad j\to \infty.
\end{equation}
For $(\xi,\tau)\in E_1(0,0)$ we have by \eqref{eq3.12} that
\[
Hu_j(\xi,\tau) = \sqrt{t_j}{}^{n+2} Hu(x,t).
\]
Hence for $(\xi,\tau) \in E_1(0,0)$ we have by \eqref{eq3.11} that
\begin{equation}\label{eq3.21}
Hu_j(\xi,\tau) = f_j(\xi,\tau)
\end{equation}
and for $(\xi,\tau)\in E_{1/4}(0,0)$ we have by \eqref{eq3.3} and \eqref{eq3.19} that
\begin{align}
 Hu_j(\xi,\tau) &\le \sqrt{t_j}{}^{n+2} \left(u(x,t) + \sqrt{\frac43}^{\, n} b \frac{1}{\sqrt{t_j}{}^n}\right)^{\frac{n+2}n}\notag\\
&\le \sqrt{t_j}{}^{n+2} \left(\frac{u_j(\xi,\tau) + C}{\sqrt{t_j}{}^n}\right)^{\frac{n+2}n}\notag\\
&= (u_j(\xi,\tau) + C)^{\frac{n+2}n}\notag\\
\label{eq3.22}
&=: v_j(\xi,\tau) ^{\frac{n+2}n}
\end{align}
where the last equation is our definition of $v_j$. Thus
\begin{equation}\label{eq3.23}
v_j(\xi,\tau) = u_j(\xi,\tau) + C\quad \text{for}\quad (\xi,\tau)\in E_{1/4}(0,0)
\end{equation}
where $C$ is a positive constant which does not depend on $(\xi,\tau)$ or $j$. Hence in $E_{1/4}(0,0)$ we have $Hu_j = Hv_j$ and
\[
 \left(\frac{Hv_j}{v_j}\right)^{\frac{n+2}2} = Hu_j\left(\frac{Hu_j}{v^{\frac{n+2}n}_j}\right)^{n/2} < Hu_j = f_j
\]
by \eqref{eq3.22} and \eqref{eq3.21}. Thus
\begin{equation}\label{eq3.24}
\iint\limits_{E_{1/4}(0,0)} \left(\frac{Hv_j}{v_j}\right)^{\frac{n+2}2}\, d\eta\, d\zeta\to 0 \quad \text{as}\quad j\to\infty
\end{equation}
by \eqref{eq3.14}.

Let $0 < R < 1/8$ and $\lambda>1$ be constants and let $\varphi\in C^\infty_0(B_{\sqrt{2R}}(0)\times (-2R,\infty))$ satisfy $\varphi\equiv 1$ on $E_R(0,0)$ and $\varphi\ge 0$ on ${\bb R}^n\times {\bb R}$. Then
\begin{align*}
 \iint\limits_{E_{2R}(0,0)} (Hv_j)v^{\lambda-1}_j \varphi^2\, d\xi\, d\tau &= \iint\limits_{E_{2R}(0,0)} \frac{Hv_j}{v_j} v^\lambda_j  \varphi^2\, d\xi\, d\tau\\
&\le \left(~\iint\limits_{E_{2R}(0,0)} \left(\frac{Hv_j}{v_j}\right)^{\frac{n+2}2} \, d\xi\, d\tau\right)^{\frac{2}{n+2}} \left(~\iint\limits_{E_{2R}(0,0)} (v^\lambda_j \varphi^2)^{\frac{n+2}n}\, d\xi\, d\tau\right)^{\frac{n}{n+2}}.
\end{align*}
Hence, using \eqref{eq3.24} and applying Lemma \ref{lem2.3} with $T=2R$, $B = \Omega = B_{\sqrt{2R}}(0)$, $E=E_{2R}(0,0)$, and $u=v_j$ we have
\[
 \iint\limits_{E_{2R}(0,0)} (v^\lambda_j \varphi^2)^{\frac{n+2}n} \, d\xi\, d\tau \le C\left(~\iint\limits_{E_{2R}(0,0)} v^\lambda_j \, d\xi\, d\tau\right)^{\frac{n+2}n}
\]
where $C$ does not depend on $j$. Therefore
\begin{equation}\label{eq3.25}
 \iint\limits_{E_R(0,0)} v^{\lambda\frac{n+2}n}_j \, d\xi\, d\tau \le C\left(~\iint\limits_{E_{2R}(0,0)} v^\lambda_j\, d\xi\, d\tau\right)^{\frac{n+2}n}.
\end{equation}
Starting with \eqref{eq3.17} with $q = \frac{n+1}n$ and applying \eqref{eq3.25} a finite number of times we find for each $p>1$ there exists $\vp>0$ such that the sequence $v_j$ is bounded in $L^p(E_\vp(0,0))$ and thus the same is true for the sequence $f_j$ by \eqref{eq3.22} and \eqref{eq3.21}. Thus by \eqref{eq3.16} and H\"older's inequality we have 
\begin{equation}\label{eq3.26}
 \limsup_{j\to\infty} \iint\limits_{E_\vp(0,0)} \Phi(-\eta,-\zeta) f_j(\eta,\zeta) \, d\eta\, d\zeta  < \infty
\end{equation}
for some $\vp>0$. Also by \eqref{eq3.14}
\begin{equation}\label{eq3.27}
\lim_{j\to\infty} \iint\limits_{E_1(0,0)-E_\vp(0,0)} \Phi(-\eta,-\zeta) f_j(\eta,\zeta)\, d\eta\, d\zeta = 0.
\end{equation}
Adding \eqref{eq3.26} and \eqref{eq3.27}, and using \eqref{eq3.15}, we contradict \eqref{eq3.20}.
\end{proof}

\section{Proof of Theorem \ref{thm1.3}}\label{sec4}

\indent 

In this section we prove the following theorem which clearly implies
Theorem \ref{thm1.3}.

\begin{thm}\label{thm4.1}
Suppose $u\in C^{2,1}(\ovl\Omega\times (0,2T))$ is a nonnegative solution of
\begin{equation}\label{eq4.1}
 \begin{cases}
   0\le u_t-\Delta u\le b\left(u+\frac{1}{\sqrt{t}^{n+1}}\right)^p&
\text{{\rm in} $\Omega\times (0,2T)$}\\
   u\le b&\text{{\rm on} $\partial\Omega\times (0,2T)$}
 \end{cases}
\end{equation}
where $T$ and $b$ are positive constants, $1<p<1+2/(n+1)$,
and $\Omega$ is a $C^2$ bounded
domain in ${\bb R}^n$, $n\ge 1$. Then there exists a positive constant
$C$ such that
\begin{equation}\label{eq4.2}
u(x,t)\le  C\frac{ \frac{\rho(x)}{\sqrt{t}} \wedge 1}{\sqrt{t}^{n+1}}
+ \sup_{\partial\Omega\times(0,T)} u
\qquad\text{for all }(x,t)\in\Omega\times(0,T),
\end{equation}
where $\rho(x)={\rm dist}(x,\partial\Omega)$.

\end{thm}

\begin{proof}

Suppose for contradiction that \eqref{eq4.2} does not hold. Then there
exists a sequence $\{(x_j,t_j)\} \subset \Omega\times (0,T)$ such that
$t_j\to 0$ as $j\to\infty$ and
\begin{equation}\label{eq4.3}
\frac{u(x_j,t_j) - \sup_{\partial\Omega\times(0,T)} u}
{\left(\frac{\rho(x_j)}{\sqrt{t_j}} \wedge 1\right)/\sqrt{t_j}^{n+1}}
\to\infty\qquad\text{as}\quad j\to\infty.
\end{equation}
For $(x,t) \in {\bb
  R}^n\times {\bb R}$ and $r>0$, let
\[
E_r(x,t) = \{(y,s) \in {\bb R}^n\times {\bb R}\colon \ |y-x| < \sqrt
r\quad \text{and}\quad t-r<s<t\}.
\]
In what follows the variables $(x,t)$ and $(\xi,\tau)$ are related by
\begin{equation}\label{eq4.4}
 x=x_j+\sqrt{t_j}\xi\quad \text{and}\quad t=t_j+t_j\tau
\end{equation}
and the variables $(y,s)$ and $(\eta,\zeta)$ are related by
\begin{equation}\label{eq4.5}
 y=x_j+\sqrt{t_j}\eta\quad \text{and}\quad s=t_j+t_j\zeta.
\end{equation}
For each
positive integer $j$, define 
\begin{equation}\label{eq4.6}
 \rho_j(\eta)=\frac{\rho(y)}{\sqrt{t_j}} \quad \text{and} \quad
f_j(\eta,\zeta) = \sqrt{t_j}^{n+3} Hu(y,s) \quad \text{for } 
(y,s)\in \ovl\Omega\times(0,2T)
\end{equation}
and define
\begin{equation}\label{eq4.7}
u_j(\xi,\tau)=\sqrt{t_j}^{n+1}
  \iint\limits_{E_{t_j}(x_j,t_j)\cap(\Omega\times(0,T))} 
G(x,y,t-s) Hu(y,s) \,dy\,ds \quad \text{for } 
(x,t)\in \ovl\Omega\times(0,2T)
\end{equation}
where $Hu$ and $G$ are as in Lemma \ref{lem2.2} and we define 
$G(x,y,\tau)=0$ if $\tau\le 0$.

By \eqref{eq2.19} we have
\begin{equation}\label{eq4.8}
  \iint\limits_{E_{t_j}(x_j,t_j)\cap(\Omega\times(0,T))} 
\rho(y) Hu(y,s)\,dy\,ds\to 0
\quad \text{as}\quad j\to\infty,
\end{equation}
and thus making the
change of variables \eqref{eq4.5} in 
\eqref{eq4.8} we get
\begin{equation}\label{eq4.9}
 \iint\limits_{E_1(0,0)\cap D_j} f_j(\eta,\zeta)\rho_j(\eta) 
\,d\eta\,d\zeta\to 0\quad \text{as}\quad j\to\infty,
\end{equation}
where $D_j=\Omega_j\times(-1,0)$ and  
$\Omega_j=\{\eta:y\in\Omega\}$. 

Since, by \eqref{eq2.27} and \eqref{eq4.6},  
\begin{align*}
G(x,y,t-s) &\le \left(\frac{\rho(x)}{\sqrt{t-s}}\wedge 1\right)
\left(\frac{\rho(y)}{\sqrt{t-s}}\wedge 1\right)
\widehat G(|x-y|,t-s)\\
&=\left(\frac{\rho_j(\xi)}{\sqrt{\tau-\zeta}}\wedge 1\right)
\left(\frac{\rho_j(\eta)}{\sqrt{\tau-\zeta}}\wedge 1\right)
\frac{1}{\sqrt{t_j}^n}
\widehat G(|\xi-\eta|,\tau-\zeta),
\end{align*}
it follows from \eqref{eq4.7} and \eqref{eq4.6} that 
for $(\xi,\tau)\in\Omega_j\times(-1,0]$ we have
\begin{equation}\label{eq4.10}
u_j(\xi,\tau)\le \iint\limits_{E_1(0,0)\cap D_j}
\left(\frac{\rho_j(\xi)}{\sqrt{\tau-\zeta}}\wedge 1\right) 
\left(\frac{\rho_j(\eta)}{\sqrt{\tau-\zeta}}\wedge 1\right)
\widehat G(|\xi-\eta|, \tau-\zeta) f_j(\eta,\zeta) 
\,d\eta\,d\zeta
\end{equation}
where we define $\widehat G(r,\tau) = 0$ if $\tau\le
0$.
It is easy to check that for $1 < q < \frac{n+2}{n+1}$ and $(\xi,\tau)
\in {\bb R}^n\times (-1,0]$ we have
\begin{equation}\label{eq4.11}
\left(~\iint\limits_{{\bb R}^n\times (-1,0)}
  \left(\frac1{\sqrt{\tau-\zeta}} 
\widehat G (|\xi-\eta|, \tau-\zeta)\right)^q d\eta\,
d\zeta\right)^{\frac1q} < C(n,q,\Omega,T) < \infty.
\end{equation} 
Thus, for $1 < q < \frac{n+2}{n+1}$, we have by \eqref{eq4.10} and
standard $L^p$
estimates for the
convolution of two functions that
\begin{equation}\label{4.12}
\|u_j\|_{L^q(E_1(0,0)\cap D_j)} \le C(n,q,\Omega,T)\|f_j\rho_j\|_{L^1(E_1(0,0)\cap
  D_j)} \to 0 \quad
\text{as}\quad j\to \infty
\end{equation}
by \eqref{eq4.9}.

If
\begin{equation}\label{eq4.13}
  (x,t) \in \ovl{E_{t_j/4}(x_j,t_j)}\cap(\Omega\times(0,T))\quad \text{and}
\quad (y,s)\in \Omega \times (0,t) - E_{t_j}(x_j,t_j)
\end{equation}
then 
\begin{equation}\label{eq4.14}
 |x-y| \ge \sqrt{t_j}/2
\end{equation}
and hence by \eqref{eq2.27} we have
\begin{align*}
  G(x,y,t-s) &\le \left(\frac{\rho(x)}{\sqrt{t-s}}\wedge 1\right)
\frac{\rho(y)}{\sqrt{t-s}} 
\widehat G\left(\frac{\sqrt{t_j}}2, t-s\right)\\
  &\le \rho(y) \max_{0<\tau<\infty} 
\left(\frac{\rho(x)}{\sqrt{\tau}}\wedge 1\right)
\frac1{\sqrt\tau} \widehat G
  \left(\frac{\sqrt{t_j}}2, \tau\right) 
= \frac{C(n,\Omega,T)\rho(y)}{\sqrt{t_j}^{n+1}}
\left(\frac{\rho(x)}{\sqrt{t_j}}\wedge 1\right).
\end{align*}
Thus for $(x,t) \in \ovl{E_{t_j/4}(x_j,t_j)}\cap(\Omega\times(0,T))$ we have
\[
\iint\limits_{\Omega\times(0,t)- E_{t_j}(x_j,t_j)} G(x,y,t-s)
Hu(y,s)\,dy\,ds \le \frac{C(n,\Omega,T)}{\sqrt{t_j}^{n+1}} 
\left(\frac{\rho(x)}{\sqrt{t_j}}\wedge 1\right)
\int\limits_{\Omega\times(0,T)}
\rho(y) Hu(y,s)\,dy\,ds.
\]
It follows therefore from Lemma \ref{lem2.2} and \eqref{eq4.7} that 
\begin{equation}\label{eq4.15}
u(x,t)\le  \frac{
u_j(\xi,\tau) + C
\left(
\frac{\rho(x)}{\sqrt{t_j}} 
\wedge 1 \right)  }
{\sqrt{t_j}^{n+1}}
+ \sup_{\partial\Omega\times(0,T)} u
\qquad
\text{for } (x,t)\in
\ovl{E_{t_j/4}(x_j,t_j)}\cap(\Omega\times(0,T))
\end{equation}
where $C$ is
a positive constant which does not depend on $j$ or $(x,t)$.

Substituting $(x,t) = (x_j,t_j)$ in \eqref{eq4.15} and using
\eqref{eq4.3} we obtain
\begin{equation}\label{eq4.16}
\frac{u_j(0,0)}{\rho_j(0)\wedge 1} \ge
\frac{u(x_j,t_j) - \sup_{\partial\Omega\times(0,T)}u}
{\left(\frac{\rho(x_j)}{\sqrt{t_j}}\wedge 1\right)/\sqrt{t_j}^{n+1}}
-C \to \infty \quad \text{as} \quad j\to \infty.
\end{equation}

For $(\xi,\tau)\in E_1(0,0)\cap D_j$ we have by \eqref{eq4.7} that
\begin{equation}\label{eq4.17}
(Hu_j)(\xi,\tau)=\sqrt{t_j}^{n+3}(Hu)(x,t).
\end{equation}
Hence for $(\xi,\tau)\in E_1(0,0)\cap D_j$ we have by \eqref{eq4.6}
that
\begin{equation}\label{eq4.18}
(Hu_j)(\xi,\tau)=f_j(\xi,\tau)
\end{equation} 
and for $(\xi,\tau)\in E_{1/4}(0,0)\cap D_j$ we have by \eqref{eq4.1}
and \eqref{eq4.15} that 
\begin{align}
Hu_j(\xi,\tau)&\le\sqrt{t_j}^{n+3}b\left(u(x,t)+\sqrt{\frac{4}{3}}^{n+1}
\frac1{\sqrt{t_j}^{n+1}}\right)^p\nonumber\\
&\le \sqrt{t_j}^{n+3}b\left(\frac{u_j(\xi,\tau)+C}
{\sqrt{t_j}^{n+1}}\right)^p\nonumber\\
&=\sqrt{t_j}^\alpha b(u_j(\xi,\tau)+C)^p
\quad \text{where } \alpha=(n+1)\left(\frac{n+3}{n+1} - p\right)>0\nonumber\\
&=:\sqrt{t_j}^\alpha bv_j(\xi,\tau)^p\label{eq4.19},
\end{align} 
where the last equation is our definition of $v_j$. Thus
\begin{equation}\label{eq4.20}
v_j(\xi,\tau)=u_j(\xi,\tau) + C
\end{equation}
where $C$ is a positive constant which does not depend on $(\xi,\tau)$  or
$j$. Hence in $E_{1/4}(0,0)\cap D_j$ we have
\[
\left(\frac{Hu_j}{v_j}\right)^{\frac{n+2}{2}} 
\le (\sqrt{t_j}^\alpha bv_j^{p-1})^{\frac{n+2}{2}}
\le \sqrt{t_j}^{\alpha(n+2)/2} b^{\frac{n+2}{2}}v_j^q,
\]
where $q=(p-1)\frac{n+2}{2} < \frac2{n+1}\frac{n+2}2 =
\frac{n+2}{n+1}$. Thus
\begin{equation}\label{eq4.21}
\iint\limits_{E_{1/4}(0,0)\cap D_j}
\left(\frac{Hu_j}{v_j}\right)^{\frac{n+2}{2}}\,d\eta\,d\zeta
\le \sqrt{t_j}^{\alpha(n+2)/2} b^{\frac{n+2}{2}}
\|v_j\|_{L^q(E_1(0,0)\cap D_j)}^q \to 0 \quad\text{as } j\to\infty 
\end{equation} 
by \eqref{4.12}.

Let $0<R<1/8$ and $\lambda>1$ be constants and let 
$\varphi\in C^\infty_0(B_{\sqrt{2R}}(0,0)\times(-2R,\infty))$ satisfy 
$\varphi\equiv 1$ on $E_R(0,0)$ and 
$\varphi\ge 0$ on ${\bb R}^n\times{\bb R}$. Then
using \eqref{eq4.20} we have
\[
v_j^\lambda \varphi^2=(u_j+C)^\lambda\varphi^2
\le 2^\lambda (u_j^\lambda\varphi^2 +C^\lambda\varphi^2)
\quad\text{in } E_{1/4}(0,0)\cap D_j
\]
and hence
\begin{align}
&\iint\limits_{E_{2R}(0,0)\cap D_j} (Hu_j) u^{\lambda-1}_j \varphi^2\,
d\xi\,d\tau 
\le \iint\limits_{E_{2R}(0,0)\cap D_j} (Hu_j) v^{\lambda-1}_j \varphi^2\,
d\xi\,d\tau\nonumber\\ 
&= \iint\limits_{E_{2R}(0,0)\cap D_j} \frac{Hu_j}{v_j} v^\lambda_j\varphi^2\,d\xi\,d\tau\nonumber\\
&\le \left(\iint\limits_{E_{2R}(0,0)\cap D_j} \left(\frac{Hu_j}{v_j}\right)^{\frac{n+2}2} 
d\xi\,d\tau\right)^{\frac2{n+2}} \left(\iint\limits_{E_{2R}(0,0)\cap D_j} (v^\lambda_j 
\varphi^2)^{\frac{n+2}n} d\xi\,d\tau\right)^{\frac{n}{n+2}}\nonumber\\
\label{eq4.22}
&\le C\left(\iint\limits_{E_{2R}(0,0)\cap D_j} \left(\frac{Hu_j}{v_j}\right)^{\frac{n+2}2} 
d\xi\,d\tau\right)^{\frac2{n+2}} \left[\left(\iint\limits_{E_{2R}(0,0)\cap D_j} (u^\lambda_j 
\varphi^2)^{\frac{n+2}n} d\xi\,d\tau\right)^{\frac{n}{n+2}} + 1\right]
\end{align}
where $C$ is a positive constant which does not depend on $j$ and
whose value may change from line to line.
Thus using \eqref{eq4.21} and applying Lemma \ref{lem2.3} with $T=2R$, 
$B=B_{\sqrt{2R}}(0)$, $E=E_{2R}(0,0)$, $\Omega=\Omega_j$, and 
$u=u_j$, we have 
\[
\left(\iint\limits_{E_{2R}(0,0)\cap D_j} (u^\lambda_j \varphi^2)^{\frac{n+2}{n}} 
d\xi\,d\tau\right)^{\frac{n}{n+2}} \le C
\left(\iint\limits_{E_{2R}(0,0)\cap D_j} 
u^\lambda_j\, d\xi\,d\tau + 1\right),
\]
Consequently, 
\begin{equation}\label{eq4.23}
\iint\limits_{E_R(0,0)\cap D_j} u^{\lambda\frac{n+2}n}_j \, d\xi\,d\tau \le 
C\left(\iint\limits_{E_{2R}(0,0)\cap D_j} 
u^\lambda_j\, d\xi\,d\tau + 1\right)^{\frac{n+2}{n}} 
\end{equation}
By \eqref{4.12},
\begin{equation}\label{eq4.24}
\lim_{j\to\infty} \iint\limits_{E_{1/4}(0,0)\cap D_j}
u^{\frac{n+3}{n+2}}_j \, d\xi\,d\tau=0.
\end{equation}
Starting with \eqref{eq4.24} and using \eqref{eq4.23} a finite number of times 
we find that
for each $p>1$ there exists $\vp>0$ such that the sequence $u_j$ is
bounded in $L^p(E_\vp(0,0)\cap D_j)$ and thus the same is true for the
sequences $v_j$, $Hu_j$, and $f_j$ by \eqref{eq4.20}, \eqref{eq4.19},
and \eqref{eq4.18}.

Thus by \eqref{eq4.10}, there exists $\vp>0$ such that
\begin{align*}
\limsup_{j\to\infty} \frac{u_j(0,0)}{\rho_j(0)} &\le \limsup_{j\to\infty} \iint\limits_{E_1(0,0)\cap D_j} \frac1{\sqrt{-\zeta}} \left(\frac{\rho_j(\eta)}{\sqrt{-\zeta}} \wedge 1\right) \widehat G (|-\eta|,-\zeta) f_j(\eta,\zeta) \, d\eta \, d\zeta\\
&\le \limsup_{j\to\infty} \left(~\iint\limits_{E_\vp(0,0)\cap D_j} \frac1{\sqrt{-\zeta}} \widehat G(|-\eta|,-\zeta) f_j(\eta,\zeta) \, d\eta\, d\zeta\right.\\ 
&\quad \left. + \iint\limits_{(E_1(0,0)-E_\vp(0,0))\cap D_j} \frac1{-\zeta} \widehat G(|-\eta|, -\zeta) f_j(\eta,\zeta) \rho_j(\eta)\, d\eta\, d\zeta\right) < \infty
\end{align*}
where we have estimated the first integral using \eqref{eq4.11} and H\"older's inequality and the second integral using \eqref{eq4.9}. Also by \eqref{eq4.10},
\begin{align*}
\limsup_{j\to\infty} u_j(0,0) &\le \limsup_{j\to\infty} \iint\limits_{E_1(0,0)\cap D_j} \left(\frac{\rho_j(\eta)}{\sqrt{-\zeta}}\wedge 1\right) \widehat G (|-\eta|,-\zeta) f_j(\eta,\zeta) \, d\eta\, d\zeta\\
&\le \limsup_{j\to\infty} \left(\iint\limits_{E_\vp(0,0)\cap D_j} 
\widehat G(|-\eta|,-\zeta) f_j(\eta,\zeta) \, d\eta \, d\zeta\right.\\
&\quad \left. + \iint\limits_{(E_1(0,0) -E_\vp(0,0))\cap D_j}
\frac1{\sqrt{-\zeta}} \widehat G(|-\eta|, -\zeta) f_j(\eta,\zeta)
\rho_j(\eta)\, d\eta \, d\zeta \right)< \infty.
\end{align*}
Hence
\[
 \limsup_{j\to\infty} \frac{u_j(0,0)}{\rho_j(0)\wedge 1} <\infty
\]
which contradicts \eqref{eq4.16} and completes the proof of Theorem~\ref{thm4.1}.
\end{proof}

\end{document}